\documentclass[12pt,reqno]{amsart}
\usepackage{amsmath, amsthm, amssymb}
\usepackage{fullpage}
\usepackage{hyperref}

\advance \topmargin by -\headheight
\advance \topmargin by -\headsep
     
\evensidemargin \oddsidemargin
\newtheorem{thm}{Theorem}[section]
\newtheorem*{thm*}{Theorem}
\newtheorem{lemma}[thm]{Lemma}

\newtheorem{cor}[thm]{Corollary}

\newtheorem*{conj*}{Conjecture}

\theoremstyle{definition}

\theoremstyle{remark}

\numberwithin{equation}{section}

\newcommand*\wrapletters[1]{\wr@pletters#1\@nil}
\def\wr@pletters#1#2\@nil{#1\allowbreak\if&#2&\else\wr@pletters#2\@nil\fi}

\def\alp{{\alpha}} 
\def\bet{{\beta}}

\def\lam{{\lambda}}

\def\le{\leqslant} \def\ge{\geqslant}

\def \bC {\mathbb C}

\def \bN {\mathbb N}

\def \bQ {\mathbb Q}
\def \bR {\mathbb R}

\def \bZ {\mathbb Z}

\def \Nm {\mathrm{Nm}}

\def \deg {\mathrm{deg}}

\begin{document}
\title[Rationality and power]{Rationality and power}
\author[Sam Chow]{Sam Chow}
\address{School of Mathematics, University of Bristol, University Walk, Clifton, Bristol BS8 1TW, United Kingdom}
\email{Sam.Chow@bristol.ac.uk}

\author[Bin Wei]{Bin Wei}
\address{School of Mathematics, Shandong University, 250100, Jinan, P. R. China}
\email{bwei.sdu@gmail.com}

\subjclass[2010]{11A99, 11D61, 11J72, 11J81}
\keywords{Elementary number theory, exponential equations, irrationality, transcendence}
\thanks{The authors thank Trevor Wooley for useful suggestions, and Adam Morgan for a historical remark. The second author is grateful to the China Scholarship Council (CSC) for supporting his studies in the United Kingdom.}
\date{}
\begin{abstract} 
We produce an infinite family of transcendental numbers which, when raised to their own power, become rational. We extend the method, to investigate positive rational solutions to the equation $x^x = \alp$, where $\alp$ is a fixed algebraic number. We then explore the consequences of $x^{P(x)}$ being rational, if $x$ is rational and $P(x)$ is a fixed integer polynomial.
\end{abstract}
\maketitle

\section{Introduction} 
\label{intro}

A famous application of the \emph{principle of excluded middle} quickly demonstrates the existence of irrational real numbers $c$ and $d$ such that $c^d$ is rational. Indeed, if $A = \sqrt{2}^{\sqrt{2}}$ is irrational then we may choose $c = A$ and $d= \sqrt{2}$, while if $A \in \bQ$ then we may choose $c=d=\sqrt{2}$. It may interest the reader to note that $A$ is transcendental, by the Gelfond-Schneider theorem \cite[Theorem 10.1]{Niv2005}.

\begin{thm*} [Gelfond, Schneider]
Let $a$ and $b$ be algebraic numbers with $a \notin \{0,1\}$ and $b \notin \bQ$. Then $a^b$ is transcendental.
\end{thm*}

Does there exist $x \in \bR \setminus \bQ$ such that $x^x \in \bQ$? As irrational powers are hard to understand, it is not immediately apparent that this problem admits a short, elementary solution. One can overcome the difficulty by using the value of $x^x$ to study $x$. In fact the equation
\[ x^x = 2 \]
has an irrational solution, since it has a real solution but no rational solution. We shall see that if $x \in \bQ_{>0}$ and $x^x \in \bQ$ then $x \in \bZ$. Consequently, any rational number $q > 1$ that is not of the shape $n^n$ $(n \in \bN)$ corresponds to an irrational number $x > 1$ such that $x^x = q$. By the Gelfond-Schneider theorem, any such $x$ is an fact transcendental. In particular, there are infinitely many transcendental real numbers $x > 1$ such that $x^x \in \bQ$.

All of the above is fairly straightforward, with the exception of the Gelfond-Schneider theorem, and these ideas have been well probed by online denizens. We generalise by considering positive rational solutions to the equation
\begin{equation} \label{MainEq}
x^x = \alp,
\end{equation}
where $\alp$ is a fixed algebraic number. We show that if $\alp$ is an algebraic integer then $x \in \bZ$. If $\alp$ is not an algebraic integer, then we are nonetheless able to bound the denominator of $x$ in terms of the degree of $\alp$. Our results provide an algorithm to determine all positive rational solutions to \eqref{MainEq}. We then show that if \eqref{MainEq} has a positive rational solution then the minimal polynomial of $\alp$ has the shape $sX^d - r$ for some $r,s,d \in \bN$. This leads to an alternate solution to the problem.

Eq. \eqref{MainEq} may have more than one positive rational solution, for instance
\[
(1/2)^{1/2} = (1/4)^{1/4}.
\]
It can have at most two solutions, however, since $x \mapsto x^x$ is strictly decreasing on $(0, 1/e]$ and strictly increasing on $[1/e, \infty)$. All positive rational solutions to the equation
\begin{equation} \label{CoolEq}
x^x = y^y
\end{equation}
are given by
\[
x = \Bigl( \frac m{m+1} \Bigr)^m, \qquad y = \Bigl( \frac m{m+1} \Bigr)^{m+1} \qquad (m \in \bN).
\]
This follows easily from a classical result describing all solutions to 
\begin{equation} \label{classical}
x^y = y^x. 
\end{equation}
Indeed, we may rearrange to see that the positive rational solutions to \eqref{CoolEq} are precisely the reciprocals of the positive rational solutions to \eqref{classical}. 

Bennett and Reznick \cite{BR2004} provide an impressive account of the history of Eq. \eqref{classical}, which goes back to a letter from Bernoulli to Goldbach \cite{Ber1843}. Bennett and Reznick credit Flechsenhaar \cite{Fle1911} as the first author to determine all positive rational solutions with proof. An industry has since developed to study equations of this type \cite{BR2004, Hau1961, Hur1967, Lau1995, Sat1972, Sve1990}. Our only addition to the historical discussion in \cite{BR2004} is to remark that Ko and Sun worked on similar problems (see \cite[pp. 113--114]{Sun1988}).

Eq. \eqref{MainEq} may be considered as a variant of the problem of investigating positive rational solutions $x$ to
$P(x^x) = 0$, for a given polynomial $P \in \bZ[X]$. Our methods allow us to draw analogous conclusions when $x^{P(x)} \in \bQ$. If $P$ is monic then $x \in \bZ$, while if $P$ is not monic then we are nonetheless able to bound the denominator of $x$ in terms of the leading coefficient of $P$.

\section{Details}
\label{details}
If $p$ is prime and $n \in \bN$, we shall write $p^a || n$ if $p^a$ divides $n$ but $p^{a+1}$ does not divide $n$. We write $[x]$ for the integer part of $x$. We denote by $[b_1, \ldots, b_s]$ the lowest common multiple of $b_1, \ldots, b_s \in \bN$. For $x \in \bR$, put $e(x) := e^{2 \pi i x}$. For $a,b \in \bZ^2 \setminus \{(0,0)\}$, we write $(a,b)$ for the greatest common divisor of $a$ and $b$.

The following lemma is crucial to our analysis.

\begin{lemma} \label{GoodLemma} Let $x,y,a,b \in \bN$ with $(a,b) = 1$, and suppose $x^a = y^b$. Then there exists $\lam \in \bN$ such that $x = \lam^b$ and $y = \lam^a$.
\end{lemma}

\begin{proof} Note that $x$ and $y$ have the same prime factors. To each such factor $p$ we assign positive integers $\alp = \alp_p$, $\beta = \beta_p$ and $X_p$ as follows. Let $p^\bet \| x$ and $p^\alp \|y$. Then $\alp/\bet = a/b$, so $b | \bet$. Write $\bet = b X_p$, and note that $\alp = aX_p$.

Put $\lam = \prod_{p|y} p^{X_p}$. Now
\[
x = \prod_{p|x} p^{\bet_p} = \prod_{p|x} p^{bX_p} = \lam^b
\]
and
\[
y = \prod_{p|y} p^{\alp_p} = \prod_{p|y} p^{aX_p} = \lam^a.
\]
\end{proof}

Our initial assertions now follow easily.

\begin{thm} \label{basic}
Let $x$ be a positive rational number, and suppose $x^x \in \bQ$. Then $x \in \bZ$.
\end{thm}

\begin{proof} Put $x = a/b$, for relatively prime positive integers $a$ and $b$. Let $x^x = m/n$, where $m$ and $n$ are relatively prime positive integers. Now 
\[
\frac{a^a}{b^a} = \frac{m^b}{n^b}.
\]
As $(a,b) = (m,n) = 1$, we must have $b^a = n^b$. Thus, by Lemma \ref{GoodLemma}, there exists $\lam \in \bN$ such that $b = \lam^b$. This is only possible if $b = \lam = 1$.
\end{proof}

\begin{cor}
Let $q > 1$ be a rational number such that $n^n = q$ has no solution $n \in \bN$. Then there exists a transcendental real number $x$ such that $x^x = q$.
\end{cor}

\begin{proof} By the intermediate value theorem, there exists a real number $x > 1$ such that $x^x = q$. The hypotheses on $q$ ensure that $x$ is not an integer, so Theorem \ref{basic} gives $x \notin \bQ$. Now the Gelfond-Schneider theorem implies that $x$ is transcendental.
\end{proof}

We turn our attention to \eqref{MainEq}. Note that if $\alp$ is not a positive real algebraic number then \eqref{MainEq} will have no solutions $x \in \bQ_{>0}$.

\begin{thm} \label{algebraic} Let $\alp$ be an algebraic number of degree $d$. Suppose \eqref{MainEq} has a solution $x = a/b$, where $a,b \in \bN$ and $(a,b) = 1$. If $\alp$ is an algebraic integer then $b=1$. Otherwise $b \ge 2$ and
\begin{equation} \label{bbound1}
\frac b{\log b} \le \frac d{\log 2}.
\end{equation}
\end{thm}

\begin{proof} From \eqref{MainEq} we have
\[
(a/b)^a = \alp^b.
\]
Applying the norm $\Nm_{\bQ(\alp)/\bQ}$ gives
\[
(a/b)^{ad} = (r/s)^b,
\]
where the minimal polynomial of $\alp$ over $\bZ$ has leading coefficient $s > 0$ and constant coefficient $\pm r$. As $(a,b) = (r,s) = 1$, we must have $b^{ad} = s^b$. Now by Lemma \ref{GoodLemma} there exists $\lam \in\bN$ such that $b^d = \lam^b$ and $s = \lam^a$. If $\alp$ is an algebraic integer then $s=1$, so $\lam = 1$ and hence $b=1$. 

If $\alp$ is not an algebraic integer then $s > 1$, so $\lam \ge 2$. We must also have $b \ge 2$, so
\[
\frac b{\log b} = \frac d{\log \lam} \le \frac d{\log 2}.
\]
\end{proof}

Note that \eqref{bbound1} implicitly bounds $b$ in terms of $d$. It is sharp, in some sense, because equality is attained when $x=1/2$. We can use \eqref{bbound1} to deduce an explicit but weaker bound.

\begin{lemma} \label{explicit} Let $b \ge 2$ and $d$ be integers satisfying \eqref{bbound1}. Then
\begin{equation} \label{bbound2}
b < 4d \log d.
\end{equation}
\end{lemma}

\begin{proof} We note that $x \mapsto x/ \log x$ is increasing on $[e, \infty)$. Using this, we deduce that
\[
\frac1{\log2} < \frac3{\log3} \le \frac2{\log2} \le \frac{n}{\log n} \qquad (n = 4,5,\ldots)
\]
Now from \eqref{bbound1} we have
\[
d \ge \frac{b \log 2}{\log b} > 1.
\]
Assume for a contradiction that \eqref{bbound2} is false. Then $b \ge 4d \log d  > e$,
so
\[
\frac b{\log b} \ge \frac{4d \log d}{\log(4d \log d)} = \frac d{\log 2} \cdot \frac{\log (d^{4 \log 2})}{\log(4d \log d)} > \frac d{\log 2},
\]
contradicting \eqref{bbound1}. This contradiction establishes \eqref{bbound2}.
\end{proof}

Based on Theorems \ref{basic} and \ref{algebraic}, as well as Lemma \ref{explicit}, we now outline an algorithm to determine all positive rational solutions to \eqref{MainEq}. Let $\alp \in \bR_{>0}$ be an algebraic number of degree $d$.
\begin{enumerate}
\item For $x = 1,2, \ldots, N$, determine whether $x^x < \alp$, $x^x = \alp$ or $x^x > \alp$. Stop as soon as $x^x \ge \alp$. The number of calculations needed here is
\[
N \le \max(3, 1 + \log \alp).
\]
If $d = 1$ then stop. 
\item Put
\[
B = [4d \log d].
\]
For $b = 2,\ldots,B$ and $a =1,2,\ldots, Nb$, test $x=a/b$ in \eqref{MainEq}. The number of tests required is at most 
\[
NB^2 = O(d^2 (\log d)^2\log(2+\alp)),
\]
with an absolute implicit constant. 
\end{enumerate}

If we know the minimal polynomial of $\alp$ over $\bZ$, then a good alternative procedure exists. We shall show that if \eqref{MainEq} has a positive rational solution then this minimal polynomial has the shape
\begin{equation} \label{MinShape}
P(X) = sX^d - r,
\end{equation}
for some $r,s,d \in \bN$.

\begin{lemma} \label{min1} 
Let $p_1, \ldots, p_s$ be distinct primes, and let $(a_1,b_1), \ldots, (a_s,b_s)$ be pairs of relatively prime integers, where $a_i \ne 0$ and $b_i > 0$ $(1 \le i \le s)$. Put $L = [b_1, \ldots, b_s]$. Then
\[
Q(X) := X^L - p_1^{a_1L/b_1} \cdots p_s^{a_sL / b_s}
\]
is irreducible over $\bQ$.
\end{lemma}

\begin{proof} Factorising $Q(X)$ over $\bC$ gives
\[
Q(X) = \prod_{j=0}^{L-1} (X-z_j),
\]
where $z_j = p_1^{a_1/b_1} \cdots p_s^{a_s/b_s} e(j/L) \quad (0 \le j \le L-1)$. Assume for a contradiction that $Q(X)$ is reducible over $\bQ$. Then there exists a nonempty proper subset $J$ of $\{0,1,\ldots,L-1\}$ such that the polynomial $\prod_{j \in J} (X-z_j)$ has rational coefficients. In particular,
\[
\prod_{j \in J} z_j \in \bQ.
\]

Let $k = \# J$, where $0 < k < L$. Now
\[
p_1^{a_1 k / b_1} \cdots p_s^{a_s k / b_s} e \Bigl(\sum_{j \in J} j/L \Bigr) \in \bQ.
\]
This forces $e \bigl(\sum_{j \in J} j/L \bigr)$ to equal $\pm 1$. Now $b_i$ divides $k$ $(1 \le i \le s)$, so $L$ divides $k$. This contradiction confirms the irreducibility of $Q(X)$.
\end{proof}

\begin{thm} \label{min2}
Let $a,b \in \bN$ with $(a,b) = 1$. Write
\[
a = p_1^{\alp_1} \cdots p_s^{\alp_s}, \qquad
b = q_1^{\bet_1} \ldots q_t^{\bet_t},
\]
where $s, t \in \bZ_{\ge 0}$, the primes $p_1, \ldots, p_s, q_1, \ldots, q_t$ are distinct, and the $\alp_i, \bet_j$ are positive integers. Put
\[
g = \gcd(b, \alp_1, \ldots, \alp_s, \bet_1, \ldots, \bet_s).
\]
Then the minimal polynomial of $(a/b)^{a/b}$ over $\bZ$ is 
\[
P(X) = b^{a/g} X^{b/g} - a^{a/g}.
\]
\end{thm}

\begin{proof} Let $y= (a/b)^{a/b}$. Since $y$ is a root of the integer polynomial $P(X)$, and since $(a,b) = 1$, it remains to show that the degree of $y$ is $b/g$. We compute
\begin{align*}
y &= p_1^{\alp_1 a/b} \cdots p_s^{\alp_s a/b} q_1^{-\beta_1 a/b} \cdots q_t^{-\bet_t a/b} \\
&= p_1^{a_1/b_1} \cdots p_s^{a_s/b_s} q_1^{-c_1 / d_1} \cdots q_t^{-c_t/d_t},
\end{align*}
where
\[
a_i = \frac{\alp_i a}{(\alp_i, b)}, \qquad b_i = \frac b{(\alp_i, b)} \qquad (1 \le i \le s)
\]
and
\[
c_i = \frac{\bet_i a}{(\bet_i, b)}, \qquad d_i = \frac b{(\bet_i, b)} \qquad (1 \le i \le t).
\]
By Lemma \ref{min1}, the degree of $y$ is
\[
[b_1, \ldots, b_s, d_1, \ldots, d_t] = b/g.
\]
For the last equality, note that if $m$ is a common multiple of $b_1, \ldots, b_s, d_1, \ldots, d_t$ then $b/(b,m)$ is a common divisor of $b, \alp_1, \ldots, \alp_s, \bet_1, \ldots, \bet_t$.
\end{proof}

Theorem \ref{min2} tells us that if \eqref{MainEq} has a positive rational solution then the minimal polynomial of $\alp$ over $\bZ$ is given by \eqref{MinShape} for some $r,s,d \in \bN$. From the proof of Theorem \ref{algebraic}, we see that $r^b = a^{ad}$, and that there exists $\lam \in \bN$ such that $\lam^a = s$. We can therefore solve \eqref{MainEq} by trying each positive divisor of $s$ as a possible value for $\lam$, since this would then determine $a$ and $b$. The time taken would be roughly the time needed to prime factorise $s$.

Finally, we have the following analogy to Theorem \ref{algebraic}. 

\begin{thm} \label{PolynomialPower}
Let $P \in \bZ[x]$ be a non-constant polynomial with leading coefficient $A$. Let $a,b \in \bN$ with $(a,b) =1$, and suppose that 
\[
(a/b)^{P(a/b)} \in \bQ.
\]
If $|A| = 1$ then $b = 1$. Otherwise
\begin{equation} \label{bbound3}
b < 3|A| \log_2 |A|.
\end{equation}
\end{thm}

\begin{proof} 
Let $d = \deg P$, and note that $b^d P(a/b) \in \bZ$. For some integer $Q = Q(a,b)$ we have
\begin{equation} \label{id}
b^d P(a/b) = Aa^d + bQ.
\end{equation}
We first suppose that $P(a/b) = 0$. In this case
\[
\frac{a^d}b = -\frac Q A, 
\]
so $b \le |A|$, and the result follows. Henceforth $P(a/b) \ne 0$.

Let $m$ and $n$ be relatively prime positive integers such that 
\[
(a/b)^{P(a/b)} = m/n.
\]
As $(a,b) = (m,n) = 1$, we have
\[
b^{|b^d P(a/b)|} = 
\begin{cases}
n^{b^d}, & \text{if } P(a/b) > 0 \\
m^{b^d}, & \text{if } P(a/b) < 0.
\end{cases}
\]
Recalling \eqref{id} now gives
\[
b^{|Aa^d + bQ| / (b^d, Aa^d + bQ)} = 
\begin{cases}
n^{b^d / (b^d, Aa^d + bQ)}, & \text{if } P(a/b) > 0 \\
m^{b^d / (b^d, Aa^d + bQ)}, & \text{if } P(a/b) < 0,
\end{cases}
\]
so by Lemma \ref{GoodLemma} there exists $\lam \in \bN$ such that
\[
b = \lam^{b^d / (b^d, Aa^d + bQ)}.
\]
Moreover,
\[
(b^d, Aa^d + bQ) \le (b, Aa^d + bQ)^d = (b,A)^d \le |A|^d.
\]

If $|A| = 1$ then $b = \lam^{b^d} \ge \lam^b$, so $\lam = b = 1$. Next consider $|A| \ge 2$. We may assume without loss that $b \ge |A|$. Now $\lam \ge 2$, so
\[
b \ge 2^{(b/|A|)^d} \ge 2^{b/|A|}.
\]
Thus
\begin{equation} \label{bbound4}
\frac b {\log_2 b} \le |A|.
\end{equation}

Recall that $x \mapsto x/ \log x$ is increasing on $[e, \infty)$. Suppose for a contradiction that \eqref{bbound3} is false. Then $b \ge 3|A| \log_2 |A| > e$, so
\[
\frac b {\log_2 b} \ge \frac{3|A| \log_2 |A|} { \log_2 (3|A| \log_2 |A|)}
 = |A| \cdot \frac {\log_2(|A|^3)} {\log_2 (3|A| \log_2 |A|)} > |A|,
\]
contradicting \eqref{bbound4}. This contradiction establishes \eqref{bbound3}.
\end{proof}

\bibliographystyle{amsrefs}
\providecommand{\bysame}{\leavevmode\hbox to3em{\hrulefill}\thinspace}

\end{document}